\documentclass{arxiv}

\newtheorem{theorem}{Theorem}[section]
\newtheorem{lemma}[theorem]{Lemma}
\newtheorem{prop}{Proposition}
\theoremstyle{definition}

\newtheorem{example}[theorem]{Example}

\theoremstyle{remark}

\numberwithin{equation}{section}

\begin{document}

 \title[Rationality of meromorphic functions between real algebraic sets]{Rationality of meromorphic functions between real algebraic sets in the plane}


\author{Tuen-Wai Ng}
\address{Department of Mathematics, The University of Hong Kong, Pokfulam, Hong Kong}
\curraddr{}
\email{ntw@maths.hku.hk}
\thanks{Tuen-Wai Ng was partially supported by the RGC grant 17306019.}

\author{Xiao Yao}
\address{School of Mathematical Sciences, Nankai University, P.R. China}
\curraddr{}
\email{yaoxiao@nankai.edu.cn}
\thanks{Xiao Yao was supported by the National Natural Science Foundation of China under Grant No. 11901311  and partially supported by National Key R\&D Program of China (2020YFA0713300).}

\subjclass[2010]{30C99, 30D99}

\date{}

\dedicatory{}


\keywords{Real algebraic curve, meromorphic function, Schwarz function.}
\begin{abstract}
We study one variable meromorphic functions mapping a planar real algebraic set $A$ to another real algebraic set in the complex plane. By using the theory of Schwarz reflection functions, we show that for certain $A$, these meromorphic functions must be rational. In particular, when $A$ is the standard unit circle, we obtain an one dimensional analog of Poincar\'e(1907), Tanaka(1962) and Alexander(1974)'s rationality results for $2m-1$ dimensional sphere in $\mathbb{C}^m$ when $m\ge 2$.
\end{abstract}

\maketitle


\bibliographystyle{amsplain}
\renewcommand{\thefootnote}{}
\section{Introduction}
In 1907, Poincar\'e \cite{Poincare1907} proved that if a biholomorphic map defined in an open set in $\mathbb{C}^2$ maps an open piece of a three dimensional sphere into another,
it is necessarily a rational map. This result was extended by Tanaka \cite{Tanaka} and then Alexander \cite{alexander1974} (for holomorphic maps) to real spheres in higher dimensions and Forstneri{\v{c}} \cite{forstnerivc1989} gave a uniform bound on the degree of the rational maps. Webster \cite{webster} then discovered a general algebraicity phenomenon for holomorphic mappings between open pieces of algebraic Levi nondegenerate hypersurfaces in the complex Euclidean $m$-space $\mathbb{C}^m$, $m \ge 2$. Webster's result leads to the following general problem:\\

{\it Under what conditions must a  holomorphic mapping
$f:\mathbb{C}^m \to \mathbb{C}^n$ sending a real algebraic set $A \subset \mathbb{C}^m$ onto another real algebraic set $A' \subset \mathbb{C}^n$ be algebraic ?} \\

Here, a subset $A \subset \mathbb{C}^m$ is a {\it real algebraic set} if it is defined by the vanishing of real-valued polynomials in $2m$ real variables and we shall always assume that $A$ is irreducible. The definition of algebraicity of holomorphic mappings can be found for example in
\cite{baouendi1996}. \\

When $m,n \ge 2$, there is a substantial literature related to the above problem, see for example, \cite{alexander1974},\cite{pinvcuk1975},
 \cite{forstnerivc1989}, \cite{huang1994},\cite{baouendi1996}, \cite{huang1998},\cite{merker2001}
 and \cite{mir2021}. In this paper, we will investigate this problem for the case $m=n=1$ which to the best of our knowledge, has not been studied in the literature, may be due to the fundamental difference between the geometric complex analysis in dimension $>1$ and the one-variable theory (see the comments in page 826 of \cite{coupet2005}). To be more precise, suppose $f$ is a meromorphic function in the complex plane $\mathbb{C}$ and $\mathcal{C}$ is a real algebraic curve in $\mathbb{R}^2 \cong \mathbb{C}$. We would like to know, for a real algebraic set $A \subset \mathcal{C}$, when $f(A)$ can lie in a real algebraic curve in $\mathbb{R}^2$. For the case that $f$ is a polynomial or rational function, it is known that $f(\mathcal{C})$ is lying in some real algebraic curve in $\mathbb{R}^2$. The following rationality result shows that $f$ has to be a rational function if $f(\mathcal{C})$ is assumed to be real algebraic and $\mathcal{C}$ is the unit circle $\mathbb{S}^1=\{z\in\mathbb{C}:|z|=1\}$. Moreover, the second part of the result can be considered as a one dimensional analog of Poincar\'e and Tanaka's rationality results (\cite{Poincare1907} and \cite{Tanaka}).

\begin{theorem}\label{thm-circle-case}
Let $f$ be a non-constant meromorphic function on $\mathbb{C}$ and $\mathbb{S}^1$ be the standard unit circle. Let $A$ be a non-degenerate continuum of $\mathbb{S}^1$. Assume that $f(A) \subset \mathcal{C}'$ where $\mathcal{C}'$ is a real algebraic curve defined by an irreducible real polynomial $P(x,y)$, then $f$ must be a rational function. Moreover, if $\mathcal{C}'=\mathbb{S}^1$, then the $f$ must be a quotient of two finite Blaschke products.
\end{theorem}

Notice that in Poincar\'e, Tanaka and Alexander's rationality results mentioned above, the mappings involved are only required to be biholomorphic or simply holomorphic in an open set containing a piece of a $2m-1$ dimensional sphere in $\mathbb{C}^m$. On the other hand, Theorem \ref{thm-circle-case} requires the map to be meromorphic in the whole complex plane. This is because the rationality result is not true if the map is defined only in an open set containing a piece of the circle $\mathbb{S}^1$ as can be seen from the following

\begin{example}\label{C-0}
Let $f_1(z)=\exp(z+\frac{1}{z})$ and $f_2(z)=\exp(-iz-\frac{i}{z})$. Then we have $f_1(\mathbb{S}^{1})$ lies in the real axis and $f_{2}(\mathbb{S}^{1})$ lies in the unit circle. Notice that $f_1$ and $f_2$ are holomorphic in $\mathbb{C}\backslash\{0\}$ and have an essential singularity at $0$.
\end{example}

The main tool of proving Theorem \ref{thm-circle-case} and other results in this paper are the Schwarz functions coined by Davis in \cite{Davis} (or the Schwarz reflection functions by Davis and Pollak \cite{DP-1958}). One can also find a generalization of Schwarz function to higher dimensions in \cite{shapiro1992}.\\

\noindent
{\bf Definition}. Let $\Gamma$ be a (non-singular) real-analytic Jordan arc or closed curve in $\mathbb{R}^2 \cong \mathbb{C}$. Suppose around  a point $a \in \Gamma$, $\Gamma$ can be represented in the form $F(z,\overline{z}) = 0$, where $F(z, w)$ is a holomorphic function of two
variables with $\frac{dF}{dw} \neq 0$. Then by the implicit function theorem, there is a unique (holomorphic) function $w = S_{\Gamma}$ in a neighborhood $N$ of $a$ such that $F(z, S_{\Gamma}(z))\equiv 0$ in $N$ and hence $S_{\Gamma}(z)=\overline{z}$ for $z$ on $\Gamma$. Moving along $\Gamma$, one obtains a
unique holomorphic function $S_{\Gamma}(z)$ in a neighborhood of $\Gamma$ which takes
the value $\overline{z}$ for $z$ on $\Gamma$. The function  $S_{\Gamma}$ or simply $S$ is called the {\it Schwarz function} or {\it Schwarz reflection function} for $\Gamma$.\\

From Chapter 6 of \cite{Davis}, we know that if $\Gamma$ is an arc of a circle with center at $z_0$ and radius $r$, then $S_{\Gamma}(z)=\overline{z_0}+\frac{r^2}{z-z_0}$. When $\Gamma$ is an arc of a straight line passing through $z_1$ and $z_2$ in $\mathbb{C}$, then $S_{\Gamma}(z)=\frac{\overline{z_1}-\overline{z_2}}{z_1-z_2}
(z-z_2)+\overline{z_2}$. These are the only possibilities for $S_{\Gamma}(z)$ to be a rational function (\cite{DP-1958} or pages 104-106 of \cite{Davis}).\\

We remark that Schwarz functions may have a branch point at infinity. If there exists no branch point at infinity, Shapiro \cite{shapiro1984} and Millar \cite{millar1990} have proven that if $\Gamma$ is a simple analytic closed curve such that $S_{\Gamma}$ is analytic outside $\Gamma$ (which is also analytic at infinity), then:
\begin{enumerate}
\item[(i)] when $S_{\Gamma}(z) = O(z)$ as $z \to \infty$, $\Gamma$ is an ellipse;
\item[(ii)] when $S_{\Gamma}(z) = O(1)$ as $z \to \infty$, $\Gamma$ is a circle;
\item[(iii)] when $S_{\Gamma}(z) = O(\frac{1}{z})$ as $z \to \infty$, $\Gamma$ is a circle centred at the origin.
\end{enumerate}

For a real algebraic planar curve $\mathcal{C}$, by construction $S_{\mathcal{C}}$ is an algebraic function (see Proposition \ref{prop-classification}). Then there exists a slit or simple curve $\gamma=\gamma(t): [0, \infty)\rightarrow\mathbb{C}$ tending to $\infty$ as $t$ goes to infinity  and each analytic branch of the Schwarz function $S_{\mathcal{C}}$ exists in $\mathbb{C}\backslash \gamma$. Here, we slightly abuse the notation and we still denote each branch of $S_{\mathcal{C}}$ in $\mathbb{C}\backslash \gamma$ by $S_{\mathcal{C}}$. Moreover, by the theory of Puiseux series, we have
$$
\lim_{z\rightarrow\infty, z\notin \gamma}S_{\mathcal{C}}(z)=c
$$
for some $c \in\mathbb{C}\cup\{\infty\}$.

From the expression of the Schwarz function for a circle, it is then clear that the following result generalizes the first part of Theorem \ref{thm-circle-case} from the unit circle $\mathbb{S}^1$ to a more general real planar algebraic curve $\mathcal{C}$.

\begin{theorem}\label{general curve}
Let $f$ be a meromorphic function on $\mathbb{C}$ and $\mathcal{C}$
be a real algebraic curve in $\mathbb{R}^2$. Assume that for some branch of $S_{\mathcal{C}}$,
\begin{equation}\label{condition a}
\lim_{z\rightarrow\infty, z\notin \gamma}S_{\mathcal{C}}(z)=c
\end{equation}
for some slit $\gamma$ and some $c\in \mathbb{C}$. Let $A$ be a non-degenerate continuum of $\mathcal{C}$. If $f(A)$ is contained in a real algebraic curve $\mathcal{C}'$ in $\mathbb{R}^2$, then $f$ must be a rational function.
\end{theorem}
\begin{example}
 From \cite{Davis}, we know that the  Schwarz functions of the family of Rose curves $$R_{2m}: r^{2m}=a+b\cos 2m \theta, 0<|b|<a, m=1, 2, \dots$$
are given by
\begin{equation*}
S_{R_{2m}}(z)=z\left(\frac{a+\sqrt{a^2-b^2+2bz^{2m}}}{2z^{2m}-b}\right)^{\frac{1}{m}},
\end{equation*}
which satisfies the condition (\ref{condition a}) in Theorem \ref{general curve}.
\end{example}

Recall that the Schwarz function for a line passing through $z_1$ and $z_2$ in $\mathbb{C}$ is given by $S(z)=\frac{\overline{z_1}-\overline{z_2}}{z_1-z_2}(z-z_2)+\overline{z_2}$ which does not satisfy condition (\ref{condition a}). The following example shows that condition (\ref{condition a}) in Theorem \ref{general curve} is essential.

\begin{example}\label{exp}
Notice that $S_{\mathbb{R}}(z)=z$. Therefore any transcendental meromorphic function which maps the real axis into itself will be an example showing that condition (\ref{condition a}) is essential. For example, the exponential function $\exp(z)$ maps the real axis to $(0,+\infty)$ which is contained in the real algebraic curve $y=0$. For another example, let $\wp$ be the Weierstrass elliptic function which satisfies the differential equation $(\wp')^2=4\wp^3-g_2\wp-g_3$. By Theorem 3.16.2 in \cite{jones1987}, if $g_2,g_3 \in \mathbb{R}$, then $\wp$ is a real function, i.e. $\wp(\overline{z})=\overline{\wp(z)}$. Hence $\wp$ maps the real axis to the real axis.
\end{example}

The functions in Example \ref{exp} belong to the so-called class $W$ (like Weierstrass),
which consists of elliptic functions,
rational functions of one exponential $\exp(\alpha z), \alpha \in\mathbb{C}\backslash\{0\}$
and rational functions in $z$. These are the only meromorphic functions which satisfy the addition formulae. They are also the only meromorphic solutions of certain non-linear complex differential equations (see \cite{eremenko2009},\cite{conte2010}, \cite{conte2012}, \cite{conte2022}, \cite{ng2019} and \cite{gu2020}).

From Examples \ref{exp}, one may ask if there exists some non-rational functions in class $W$ that can map real irreducible planar algebraic sets (other than the straight lines) to another real planar algebraic set. The following result shows that this is impossible.

\begin{theorem}\label{T3}
Let $f$ be a transcendental meromorphic function in class $W$ and $\mathcal{A}$, $\mathcal{B}$ be two real planar algebraic curves. Let $A$ be a non-degenerate continuum of $\mathcal{A}$. Assume  that $f(A)\subset \mathcal{B}$, then $A$ is a straight line segment.
\end{theorem}

Theorem \ref{T3} shows that for any transcendental $f$ in class $W$, $f(\mathbb{S}^1)$ cannot be part of a real planar algebraic curve.

\section{Proofs of Theorems}

\noindent
{\it Proof of Theorem \ref{thm-circle-case}}.
Theorem \ref{general curve} implies that $f$ is rational. Here we also include an elementary proof for the case of $\mathbb{S}^1$.  We may assume that $0$ is not a pole of $f$ for otherwise we can replace $f$ by $\frac{1}{f}$ and $A$ by the set $\{\frac{1}{z}: z \in A\}$ in the following arguments. Note that $\mathcal{C}'$ has finitely many singular point and $f$ is non-constant meromorphic. So we can find some open arc $L \subset A$ such that $L$ contains no poles and critical points of $f$ and hence $B=f(L)$ is an analytic arc which contains no singular points of $\mathcal{C}'$.

Let $S_L$ and $S_B$ be the Schwarz functions associated with $L$ and $B$ respectively. Then from (9.3) of Davis and Pollak's paper \cite{DP-1958}, we have
$$f(\overline{S_L(z)})=\overline{S_B(f(z))}$$ in a neighborhood $N$ of $L$. Notice that as $L \subset \mathbb{S}^1$, its Schawrz function is $S_L(z)=\frac{1}{z}$. Hence we have in $N$,
$$\overline{f(\frac{1}{\overline{z}})}=S_B(f(z)).$$ Since $f$ is meromorphic in $\mathbb{C}$, $g(z):=\overline{f(\frac{1}{\overline{z}})}$ is meromorphic in $\mathbb{C}\backslash\{0\}$ and it follows that $S_B(f(z))$ is meromorphic in $N$.

To get the Schwarz function of $B$, we rewrite $P(x, y)= 0$ as $Q(z,\overline{z})=0$ for some $Q \in \mathbb{C}[z,\overline{z}]$. Then $f(L) \subset \mathcal{C}'$ implies that $$Q(f(z),S_B(f(z))\equiv 0$$ in $N$. Thus we have $Q(f(z),g(z))\equiv 0$ in $N$ and hence in the (connected) subset of $\mathbb{C}\backslash\{0\}$ where both $f$ and $g$ are analytic.

If $Q$ depends on $z$ only, then $f$ will be a single valued algebraic function in $\mathbb{C}$ and hence a constant which is impossible. So we may let  $Q(f,g)=q(f)g^n + \cdots$ where $n \ge 1$ and $q$ is a polynomial. We first consider the case $q(f(0))\neq 0$. Then we have $0<c<|q(f(z))|<d$ in some closed disk $\overline{\mathbb{D}(0;r)}$. By choosing $r$ sufficiently small, we also have $Q(f(z),g(z))\equiv 0$ in $\overline{\mathbb{D}(0;r)}\backslash \{0\}$. Then it follows from Lemma \ref{lemma-key} below that $|g(z)|$ is bounded above in $\overline{\mathbb{D}(0;r)}\backslash \{0\}$ and hence $g$ has a removable singularity at $0$. This will force $f$ to have a removable singularity  at infinity so that $f$ must be rational.

\begin{lemma}\cite[Theorem 8.13, Corollary 8.1.8]{RS-2002}\label{lemma-key}
Let $p(z) =a_nz^n+\cdots+a_1z+a_0$ be a polynomial of degree $n$. Then all the zeros of $p$ lie in the closed disk $\overline{\mathbb{D}(0;\rho)}$ where $$\rho \le \max_{0 \le i \le n-1} \Big(n|\frac{a_i}{a_n}|\Big)^{\frac{1}{n-i}}$$
\end{lemma}

Now suppose $q(f(0))=0$, then we can write $q(f(z))=z^mf_0(z)$ where $m \in \mathbb{N}$, $f_0$ is meromorphic in $\mathbb{C}$ with $f_0(0) \neq 0$. Then we may assume that $0<c_0<|f_0(z)|<d_0$ in some closed disk $\overline{\mathbb{D}(0;r_0)}$ with $r_0 <1$. We will also let $f(z)=z^kf_1(z)$ where $k \ge 0$ and $f_1$ is meromorphic in $\mathbb{C}$ with $f_1(0) \neq 0$. By Lemma \ref{lemma-key}, we have for $z\in \overline{\mathbb{D}(0;r_0)}$,
$$|g(z)| \le \frac{C}{|z|^{m}}$$
for some positive constant $C$. Since $g(z)=\frac{1}{z^k}\overline{f_1(\frac{1}{\overline{z}})}$, it follows that
$$|\overline{f_1(\frac{1}{\overline{z}})}| \le \frac{C}{|z|^{m-k}}$$
in $\overline{\mathbb{D}(0;r_0)}$. 
It follows that $|\frac{f_1(w)}{w^{m-k}}|\le C$ for all large $w$. Then the meromorphic function $\frac{f_1(w)}{w^{m-k}}$ has a removable singularity at infinity and hence $\frac{f_1(w)}{w^{m-k}}$ is rational. Therefore $f_1$ and hence $f$ must be rational.\\

Now if $\mathcal{C}'=\mathbb{S}^1$, then we need to show that the rational function $f$ must be a quotient of two finite Blaschke products. This will follow from the following result
of Pakovich and Shparlinski \cite[Theorem 2.2]{PS-2020}.

\begin{lemma}[\cite{PS-2020}]
Let $p_1$ and $p_2$ be two one variable complex rational functions of degrees $n_1$ and $n_2$ respectively. Then
$$|\{z \in \mathbb{C} : |p_1(z)| = |p_2(z)| = 1\}| \le (n_1 + n_2)^2,$$
unless $p_1 = b_1 \circ q$ and $p_2 =b_2 \circ q$ for some quotients of finite Blaschke products $b_1$ and $b_2$ and rational function $q$.
\end{lemma}

As $\mathcal{C}'=\mathbb{S}^1$, We can take $p_1(z)=z$, $p_2=f$ and since the infinite set $A$ is a  subset of $\{z \in \mathbb{C} : |p_1(z)| = |p_2(z)| = 1\}$, we must have $z = b_1 \circ q$ and hence the degree of the rational function $q$ is one and thus $q$ is a M\"obius transformation.
From $z=b_1 \circ q$, we deduce that $q^{-1}=b_1$ and hence $|q^{-1}(z)|=1$ whenever $|z|=1$ as $b_1$ is a quotient of finite Blaschke products. Hence $q$ must be a degree one  finite Blaschke product. Notice that $f=b_2 \circ q = \frac{B_1}{B_2} \circ q$ where $B_1$ and $B_2$ are finite Blaschke products. Since a composition of two finite Blaschke products is still a finte Blaschke product, $f$ will be a quotient of finite Blaschke products.{\hfill $\Box$\par\vspace{2.5mm}}

\medskip

 To prove Theorem \ref{general curve}, we need the following proposition.

\begin{prop}\label{prop-classification}
Let $A$ be a real algebraic set in $\mathbb{R}^2$. Then the Schwarz reflection function $S_{A}$ is an algebraic function. Moreover, if $A$ is contained in an irreducible real algebraic planar curve $\mathcal{C}$, then a branch of $S_A$ coincides with a branch of $S_{\mathcal{C}}$ in the complement of some slit in $\mathbb{C}$.
\end{prop}

\begin{proof}
The proof for $S_L$ being algebraic can be found in page 21-22 of \cite{Davis}. By the construction of Schwarz functions in page 30-33 of \cite{Davis}, it is clear that a branch of $S_L$ coincides with a branch of $S_{\mathcal{C}}$ in a neighborhood $N$ of $L$. Therefore, these two branches will coincide in any domain $D$ (containing $N$) in which both branches are analytic. In particular, we may take $D$ to be the complement of a suitable slit in $\mathbb{C}$.
\end{proof}

\noindent
{\it Proof of Theorem \ref{general curve}}.
 We denote the real algebraic curve $\mathcal{C}'$ defined by $$
\mathcal{C}'=\{
(x, y)\in \mathbb{R}^2: Q_0(x, y)=0\}
$$
where $Q_0\in \mathbb{R}[x, y]$. Recall that $f(A)\subset \mathcal{C}'$, then as in the proof of Theorem \ref{thm-circle-case}, we can choose an analytic arc $L$ (in $A$) which contains no singular point of $\mathcal{C}$ such that $B=f(L)$ is also an analytic arc and $B$ contains no singular point of $\mathcal{C}^{'}$.
We  rewrite $Q_0(x, y)$ as $Q(w, \overline{w})\in \mathbb{C}[w, \overline{w}]$ such that
\begin{equation*}
Q(f(z), S_{B}(f(z)))\equiv 0
\end{equation*}
in  some neighborhood $N$ of $L$. Since $f(L)$ is connected, we may further assume that $Q$ is irreducible, for otherwise we can consider instead an irreducible factor of $Q$.
Recall that from (9.3) of \cite{DP-1958} or (8.7') of \cite{Davis},
\begin{equation*}
f(\overline{S_{L}(z)})\equiv\overline{S_{B}(f(z))}
\end{equation*}
in $N$ and hence we have
\begin{equation}
Q(f(z), \overline{f(\overline{S_{L}(z)})})\equiv 0
\end{equation}
in $N$. By Proposition \ref{prop-classification}, we know that $S_{L}$ is an algebraic function. Therefore, there exists a simple curve $\gamma=\gamma(t): [0, \infty)\rightarrow \mathbb{C}$  tending to infinity as $t$ goes to infinity, such that $\overline{f(\overline{S_{L}(z)})}$ is meromorphic in $V=\mathbb{C}\backslash \gamma$ and hence
\begin{equation}\label{V}
Q(f(z), \overline{f(\overline{S_{L}(z)})})\equiv 0
\end{equation}
in $V\backslash Z$ where $Z$ is the set of the poles of $f$ or $\overline{f(\overline{S_{L}(z)})}$. As $S_{L}$ is an algebraic function, we can consider the Puiseux series of it at infinity to see that  there exists some $m\in \mathbb{N}$ such that  $S_{L}(z^m)$ is analytic at some punctured neighborhood $U$ of infinity.
This implies that
\begin{equation}
Q(f(z^m), \overline{f(\overline{S_{L}(z^m)})})\equiv 0
\end{equation}
in $W=\{z: z^m\in U\backslash (Z\cup\gamma)\}$.  Let $F(z)=f(z^m)$ and $G(z)=\overline{f(\overline{S_{L}(z^m)})}$ so that both are well-defined  analytic functions in $W$. Assume that $f$ is transcendental. Then $F$ is transcendental meromorphic in $\mathbb{C}$ so that by the Little Picard Theorem, $F$ has at most two Picard exceptional values. Let $a \in \mathbb{C}\backslash F(\gamma)$ be any non-Picard exceptional value of $F$. Then there exists a sequence of $\{z_n\}_{n=1}^{\infty}\subset W$ tending to infinity such that $F(z_n)=a$.

We first consider the case that $f$ is analytic in some neighborhood of $\overline{c}$. By the assumption (\ref{condition a}) and the fact that $S_{\mathcal{C}}\equiv S_L$ for some branch of $S_L$ in $\mathbb{C}\backslash \gamma$ (by Proposition \ref{prop-classification}), we have
$\lim_{z\rightarrow\infty, z\notin \gamma}S_L(z)=c$.Thus, we have
$\lim\limits_{n\rightarrow\infty}G(z_n)=\overline{f(\overline{c})}$ and
\begin{equation*}
Q(a, \overline{f(\overline{c})})=Q(F(z_n),\, \lim_{n\rightarrow\infty}G(z_n))
=\lim_{n\rightarrow\infty}Q(F(z_n), G(z_n))=0.
\end{equation*}
Since there are infinitely many $a$ satisfying $Q(a, \overline{f(\overline{c})})=0$, we have $Q(z, \overline{f(\overline{c})})\equiv0$ in $\mathbb{C}$. Hence $Q(u, v)=(v-\overline{f(\overline{c})})Q'(u, v)$ for some non-constant $Q'\in \mathbb{C}[u, v]$, which contradicts the assumption that $Q$ is irreducible in $\mathbb{C}[u, v]$.

It remains to consider the case that $\overline{c}$ is a pole of $f$. Assume that $Q(F,G)=\sum_{k=0}^{n}p_{k}(F)G^k$, where $p_k \in \mathbb{C}[z]$ for $1\leq k\leq n$ and $p_{n}\not\equiv 0$. We may also assume that $p_{n}(a)\neq 0$ as there are infinitely many non-Picard exceptional values of $F$.  
Then by Lemma \ref{lemma-key}, $\{G(z_n)\}_{n=1}^{\infty}$ will be a bounded sequence. This contradicts to the fact that
$$\lim_{n\rightarrow\infty}G(z_n)=\overline{f(\overline{c})}=\infty$$
and we conclude that $f$ must be rational.
{\hfill $\Box$\par\vspace{2.5mm}}

\noindent
{\it Proof of Theorem \ref{T3}}.
We first deal with the case that $f$ is a non-constant elliptic function with periodic lattice $\Lambda=\mathbb{Z}\omega_1\bigoplus \mathbb{Z}\omega_2$ for some $\mathbb{R}$-linearly independent $\omega_1$ and $\omega_2$ in $\mathbb{C}\backslash\{0\}$. We also assume that the finite map $f:\mathbb{C}/\Lambda \to \mathbb{CP}^{1}$ has degree $d$ so that for each $a\in \mathbb{CP}^1$, $f(z)=a$ has exactly $d$ solutions in the periodic parallelogram.

Like what we did in the proofs of Theorem \ref{thm-circle-case} and \ref{general curve}, we can assume that $A$ is an analytic arc and  there exists an analytic branch of $S_{A}$ in $\mathbb{C}\backslash \gamma$ and
an irreducible polynomial $Q$ in $\mathbb{C}[u, v]$ such that
\begin{equation}
Q(f(z), \overline{f(\overline{S_{A}(z)})})\equiv 0
\end{equation}
in the complement of the slit $\gamma$ in $\mathbb{C}$. As $f$ is non-constant elliptic, $n=\deg_vQ(u,v) \ge 1$. Let $g(z)=\overline{f(\overline{S_{A}(z)})}$.
As $\overline{f(\overline{z})}$ is meromorphic in $\mathbb{C}$ and $S_A$
 is analytic in $\mathbb{C}\backslash\gamma$, $g$ is meromorphic in $\mathbb{C}\backslash\gamma$.

 To proceed, we will make use of the finite map property of $f$ and a counting argument (see page 553 of \cite{beardon2006}). Now choose any complex number $z_0$ in $\mathbb{C}\backslash\gamma$ such that $\deg_vQ(f(z_0),v)=n$ and the orbit $z_0+\mathbb{Z}\omega_1$ is a  subset of  $\mathbb{C}\backslash\gamma$ and does not contain a pole of $f$ or $g$. Then for all $n \in \mathbb{Z}$,
\begin{equation}
Q(f(z_0), g(z_0+n\omega_1))=Q(f(z_0+n\omega_1), g(z_0+n\omega_1))=0.
\end{equation}
As $Q(f(z_0),v)$ has at most $n$ distinct zeros, there must be a  subset $M_{z_0}$ of $\{1,\ldots,nd+1\}$ with cardinality at least $d+1$  such that
$g(z_0+m_k\omega_1)=g(z_0+m_l\omega_1)$ for any $m_k,m_l \in M_{z_0}$.
As each $M_{z_0}$ is a subset of $\{1,\ldots,nd+1\}$, there are only finitely many distinct $M_{z_0}$. Since there is an uncountable
number of choices of $z_0$, there must be an uncountable set $U$ of $z$ for which
the set $M$ is independent of $z_0$ in $U$. Since any uncountable subset of $\mathbb{C}$
has an uncountable set of accumulation points, we deduce that
for any $m_k$ and $m_l$ in $M$, $g(z+m_k\omega_1)\equiv g(z+m_l\omega_1)$ in $\mathbb{C}\backslash\gamma$ and hence we have
\begin{equation}\label{elliptic}
f(\overline{S_{A}(z+m_k\omega_1)})\equiv f (\overline{S_{A}(z+m_l\omega_1)})
\end{equation}
in $\mathbb{C}\backslash\gamma$.
As $f:\mathbb{C}/\Lambda \to \mathbb{C}\cup\{\infty\}$ is of degree $d$ and $M$ contains at  least $d+1$ points, it follows from (\ref{elliptic}) and a similar  counting argument that there exist distinct $n_1, n_2 \in M$ such that
$$
\overline{S_{A}(z+n_1\omega_1)}-\overline{S_{A}(z+n_2\omega_1)} \in \Lambda
$$
for $z \in \mathbb{C}\backslash\gamma$.
As $\Lambda$ is a discrete set and $S_{A}(z+n_1\omega_1)-S_{A}(z+n_2\omega_1)$ is analytic in the connected set $\mathbb{C}\backslash\{(\gamma-n_1\omega_1)\cup(\gamma-n_2\omega_1) \}$, we must have
$$
S_{A}(z+n_1\omega_1)-S_{A}(z+n_2\omega_1)\equiv
l_1\overline{\omega_1}+l_2\overline{\omega_2}
$$
in $\mathbb{C}\backslash\{(\gamma-n_1\omega_1)\cup(\gamma-n_2\omega_1) \}$ for some $l_1, l_2\in \mathbb{Z}$.
Since $S_{A}(z+n_2\omega_1)$ is analytic in $\gamma-n_1\omega_1$, it follows that $S_{A}(z+n_1\omega_1)$  can be extended to an entire function by the above identity. Since $S_A$ is algebraic and entire, it must be a polynomial. Notice that the polynomial $S_A'$ is periodic. This will reduce $S_A'$ to a constant and hence $S_A(z)=az+b$. Then from (6.11) of Davis' book \cite{Davis}, we have  $\overline{S_A(\overline{S_A(z)})}=z$, and we can deduce that $|a|=1$. To see that $A$ is a straight line segment, consider any two points $z_1,z_2$ in $A$. Then $\overline{z_i}=S_A(z_i)=az_i+b$ for $i=1,2$. It then follows that
$\frac{\overline{z_1}-\overline{z_2}}{z_1-z_2}=a$ and we are done.\\

Now consider that case that $f(z)=R(e^{\alpha z})$ for some rational function $R$  and $\alpha\in \mathbb{C}\backslash\{0\}$. Without loss of generality, we may assume that $\alpha=2\pi i$ so that $f$ has period $1$. Let $D=\{x+iy:0\le x < 1\}$. Then $f:D\to \mathbb{C}\cup\{\infty\}$ is a finite map. One can apply a similar argument to show that $A$ is a straight line segment.
{\hfill $\Box$\par\vspace{2.5mm}}


\end{document}